\documentclass[12pt,reqno]{article}

\usepackage[usenames]{color}
\usepackage[colorlinks=true,
linkcolor=webgreen, filecolor=webbrown,
citecolor=webgreen]{hyperref}

\definecolor{webgreen}{rgb}{0,.5,0}
\definecolor{webbrown}{rgb}{.6,0,0}

\usepackage{amssymb}
\usepackage{graphicx}
\usepackage{amscd}
\usepackage{lscape}
\usepackage{tikz}
\usepackage{tikz-cd}
\usepackage{pgfplots}

\usetikzlibrary{matrix}
\usetikzlibrary{fit,shapes}
\usetikzlibrary{positioning, calc}
\tikzset{circle node/.style = {circle,inner sep=1pt,draw, fill=white},
        X node/.style = {fill=white, inner sep=1pt},
        dot node/.style = {circle, draw, inner sep=5pt}
        }
\usepackage{tkz-fct}

\usepackage{amsthm}
\newtheorem{theorem}{Theorem}

\newtheorem{proposition}[theorem]{Proposition}

\theoremstyle{definition}

\newtheorem{example}[theorem]{Example}

\usepackage{float}

\usepackage{graphics,amsmath}
\usepackage{amsfonts}
\usepackage{latexsym}
\usepackage{epsf}

\DeclareMathOperator{\Rev}{Rev}
\DeclareMathOperator{\INVERT}{INVERT}
\DeclareMathOperator{\MINVERT}{MINVERT}

\newcommand{\seqnum}[1]{\href{http://oeis.org/#1}{\underline{#1}}}

\begin{document}

\begin{center}
\vskip 1cm{\LARGE\bf Moment sequences, transformations, and Spidernet graphs} \vskip 1cm \large
Paul Barry\\
School of Science\\
South East Technological University\\
Ireland\\
\href{mailto:pbarry@wit.ie}{\tt pbarry@wit.ie}
\end{center}
\vskip .2 in

\begin{abstract} We use the link between Jacobi continued fractions and the generating functions of certain moment sequences to study some simple transformations on them. In particular, we define and study a transformation that is appropriate for the study of spidernet graphs and their moments, and the free Meixner law.
\end{abstract}

\section{Introduction}

In this section we study two well-known transforms on (integer) sequences, namely the binomial and the INVERT transform, and then we define the ``mean INVERT'' transform, which turns out to be significant in the context of this note.

We consider a sequence $a_n$ with generating function $g(x)=\sum_{n=0}^{\infty} a_n x^n$. The \emph{binomial transform} of $a_n$ is the sequence $b_n$ where
$$ b_n = \sum_{k=0}^n \binom{n}{k}a_k.$$ The sequence $b_n$ then has its generating function given by
$$ \frac{1}{1-x} g\left(\frac{x}{1-x}\right).$$
More generally, the $r$-th binomial transform ($B_r$) of $a_n$ is given by
$$\sum_{k=0}^n \binom{n}{k}r^{n-k} a_k.$$ This has its generating function given by
$$\frac{1}{1-rx} g\left(\frac{x}{1-rx}\right).$$
The INVERT transform of $a_n$ is the sequence whose generating function is given by $\frac{g(x)}{1-x g(x)}$. More generally, the $r-\INVERT$ transform ($\INVERT_r$) of $a_n$ is the sequence with generating function $\frac{g(x)}{1-rxg(x)}$. Relationships between these two transformations include the following.
\begin{proposition} We have
$$B_r \circ \INVERT_s = \INVERT_s \circ B_r.$$
\end{proposition}
\begin{proof}
We have
\begin{align*}
B_r (\INVERT_s (g(x)))&= B_r (\frac{g(x)}{1-sxg(x)}\\
&=\frac{1}{1-rx} \frac{g(x)}{1-sxg(x)}\left(\frac{x}{1-rx}\right)\\
&=\frac{1}{1-rx} \frac{g\left(\frac{x}{1-rx}\right)}{1-s \frac{x}{1-rx} g\left(\frac{x}{1-rx}\right)}\\
&=\frac{\frac{1}{1-rx} g\left(\frac{x}{1-rx}\right)}{1-s \frac{x}{1-rx} g\left(\frac{x}{1-rx}\right)}\\
&=\INVERT_s (\frac{1}{1-rx} g\left(\frac{x}{1-rx}\right))\\
&=\INVERT_s (B_r (g(x)).\end{align*}
\end{proof}
We let $\Rev(f(x))=f^{\langle -1 \rangle}(x)$ denote the compositional inverse of $f(x)=\sum_{n=1}^{\infty} f_n x^n$ where $f_1 \ne 0$. For a generating function $g(x)=\sum_{n=0}^{\infty}g_n x^n$ with $g_0 \ne 0$ and $g_1 \ne 0$,  we define
$$\Rev(g)(x)=\frac{1}{x}\Rev(xg)(x).$$ Using Lagrange inversion \cite{LI}, we have, in these circumstances,
$$[x^n] \Rev(g) = \frac{1}{n+1} [x^n] \frac{1}{g(x)^{n+1}}.$$
\begin{proposition} We have
$$ \Rev(B_r(g(x))= \INVERT_{-r}(\Rev{g}(x)).$$
\end{proposition}
\begin{proof}
We have
$$\Rev(B_r(g(x))=\frac{1}{x} \Rev\left(\frac{x}{1-rx}g\left(\frac{x}{1-rx}\right)\right)=\frac{1}{x}\Rev\left((xg)\left(\frac{x}{1-rx}\right)\right).$$
Now $\Rev\left((xg)\left(\frac{x}{1-rx}\right)\right)$ is the solution $u(x)$ of the equation
$$(xg)\left(\frac{u}{1-ru}\right)=x$$ for which $u(0)=0$.
We find that
$$\frac{u}{1-ru} = (xg)^{\langle -1 \rangle}(x)=\Rev(xg)(x).$$ Solving for $u(x)$, we find that
$$u(x)=\frac{\Rev(xg)}{1+r\Rev(xg)}.$$
Thus
\begin{align*}
\Rev(B_r(g(x))&=\frac{1}{x} \Rev\left(\frac{x}{1-rx}g\left(\frac{x}{1-rx}\right)\right)\\
&=\frac{1}{x}\frac{\Rev(xg)}{1+r\Rev(xg)}\\
&=\frac{1}{x} \frac{x \Rev(g)}{1+rx\Rev{g}}\\
&= \frac{\Rev{g}}{1+rx \Rev{g}}\\
&= \INVERT_{-r}(\Rev(g)(x)).\end{align*}
\end{proof}
We now consider an average of INVERT transforms, by looking at the expression
$$\frac{1}{2}\left(\INVERT_r + \INVERT_{-r}\right)(g(x)).$$ This is equal to
\begin{align*}\frac{1}{2}\left(\INVERT_r + \INVERT_{-r}\right)(g)(x)&=\frac{1}{2}\left(\frac{g(x)}{1-r x g(x)}+\frac{g(x)}{1+rxg(x)}\right)\\
&=\frac{1}{2}\left(\frac{(1+rxg(x))g(x)+(1-rxg(x))g(x)}{(1-rxg(x))(1+rxg(x))}\right)\\
&=\frac{1}{2}\left(\frac{2 g(x)}{1-r^2x^2 g(x)^2}\right)\\
&=\frac{g(x)}{1-r^2x^2g(x)^2}.\end{align*}
We now define the ``mean INVERT'' transform of $g(x)$ to be 
$$\MINVERT_r(g)(x)= \frac{g(x)}{1-r x^2 g(x)^2}.$$ 
Note that for this definition, we find it convenient to use $r$ in the denominator, so that 
$$\MINVERT_r(g)(x)=\frac{1}{2}\left(\INVERT_{\sqrt{r}} + \INVERT_{-\sqrt{r}}\right)(g(x)).$$
\begin{example} The generating function of the Catalan numbers $C_n=\frac{1}{n+1}\binom{2n}{n}$ \seqnum{A000108} is given by
$$c(x)= \frac{1-\sqrt{1-4x}}{2x}.$$
Then we have
$$\MINVERT_1(c)(x)=\frac{c(x)}{1-x^2 c(x)^2}.$$
This expands to give the sequence \seqnum{A000958}, which begins
$$1, 1, 3, 8, 24, 75, 243, 808, 2742, 9458, 33062,\ldots.$$ This sequence counts the number of ordered rooted trees with $n$ edges having a root of odd degree.
\end{example}
\begin{example} The Motzkin numbers \seqnum{A001006} have generating function 
$$m(x)=\frac{1-x-\sqrt{1-2x-3x^2}}{2x^2}.$$ Then 
$$\MINVERT_1(m)(x)=\frac{1}{1-xm(x)^2}$$ is the generating function of the central trinomial numbers \seqnum{A002426}. The generating function 
$$\INVERT\cdot\MINVERT_1(m)(x)=\frac{1}{1-2x-2x^2m(x)}$$ is that of \seqnum{A111961}, which is a transform of the Fibonacci numbers $F_{n+1}$ by the Motzkin matrix \seqnum{A064189}.
\end{example}
Sequences in this note, where documented, are referred to by their A$nnnnnn$ number in the On-Line Encyclopedia of Integer Sequences (OEIS) \cite{SL1, SL2}.

\section{Transforming simple moment sequences}
Among the simplest moment sequences are those whose generating functions can be expressed as a constant coefficient Jacobi continued fraction \cite{Wall} of the form
$$g_{a,b}(x)=\cfrac{1}{1-ax-\cfrac{bx^2}{1-ax-\cfrac{bx^2}{1-ax-\cdots}}}.$$ These are the moments of scaled Chebyshev polynomials of the second kind, defined as follows. 
$$P_n(x)=\sqrt{b}^n U_n\left(\frac{x-a}{2\sqrt{b}}\right).$$  In this case we have
$$g_{a,b}(x)=\Rev\left(\frac{1}{1+ax+bx^2}\right),$$ and if we have $g_{a,b}(x)=\sum_{n=0}^{\infty}\mu_n x^n$,
then 
$$\mu_n=\frac{1}{\pi} \int_{a-2\sqrt{b}}^{a+2\sqrt{b}} x^n \frac{\sqrt{4b-(x-a)^2}}{2b}\,dx.$$
\begin{proposition} Let $g(x)$ have continued fraction expression
$$\cfrac{1}{1-\alpha_1 x - \cfrac{\beta_1 x^2}{1-\alpha_2 x - \cfrac{\beta_2 x^2}{1-\alpha_3x-\cdots}}}.$$
Then the generating function $B_r(g)(x)$ will have the continued fraction expression
$$\cfrac{1}{1-(\alpha_1+r) x - \cfrac{\beta_1 x^2}{1-(\alpha_2+r) x - \cfrac{\beta_2 x^2}{1-(\alpha_3+r)x-\cdots}}}.$$
\end{proposition}
If we write $g(x)=\mathcal{J}(\alpha_1, \alpha_2,\ldots; \beta_1,\beta_2,\ldots)$ then we have
$$B_r(g)(x)=\mathcal{J}(\alpha_1+r, \alpha_2+r,\ldots; \beta_1,\beta_2,\ldots).$$
Thus the binomial transform only effects the coefficients of $x$ in the continued fraction, and it does this by a simple addition to each such coefficient.
\begin{proposition}
We let $g(x)=\mathcal{J}(\alpha_1, \alpha_2,\ldots; \beta_1,\beta_2,\ldots)$. Then we have
$\INVERT_s(g)(x)=\mathcal{J}(\alpha_1+s, \alpha_2,\ldots; \beta_1,\beta_2,\ldots)$.
\end{proposition}
\begin{proof} We have 
$$g(x)=\cfrac{1}{1-\alpha_1 x - \cfrac{\beta_1 x^2}{1-\alpha_2 x - \cdots}}.$$ 
Then $$\frac{1}{g(x)}=1-\alpha_1 x - \cfrac{\beta_1 x^2}{1-\alpha_2 x - \cdots}.$$
Now we have 
\begin{align*}\frac{g(x)}{1- sxg(x)}&=\frac{1}{\frac{1}{g(x)}-sx}\\
&=\frac{1}{1-\alpha_1 x - \cfrac{\beta_1 x^2}{1-\alpha_2 x - \cdots}-sx}\\
&=\cfrac{1}{1-(\alpha_1 + s)x - \cfrac{\beta_1 x^2}{1-\alpha_2 x - \cdots}}.\end{align*}
\end{proof}
Thus the INVERT transform only effects the coefficient of the first occurrence of $x$ in the continued fraction.
\begin{proposition}
We let $g_{\alpha,\beta}(x)=\mathcal{J}(\alpha, \alpha,\ldots; \beta,\beta,\ldots)$ be a constant coefficient Jacobi continued fraction. Then we have
$\MINVERT_r(g)(x)=\mathcal{J}(\alpha_, \alpha_,\alpha,\ldots; \beta+r,\beta,\beta,\ldots)$.
\end{proposition}
\begin{proof}
We have
$$g_{\alpha,\beta}(x)=\frac{1}{1-\alpha x - \beta x^2 g_{\alpha,\beta}(x)}.$$
Then (where we write $g(x)=g_{\alpha,\beta}(x)$)
\begin{align*}
\MINVERT_r(g)(x)&=\frac{g(x)}{1-rx^2g(x)^2}\\
&=\frac{1}{1-\alpha x - \beta x^2 g(x)}\cdot \frac{1}{1-r x^2 g(x)^2}\\
&=\frac{1}{1-\alpha x- \beta x^2 g(x) -rx^2 g(x)^2 (1-\alpha x - \beta x^2 g(x))}\\
&=\frac{1}{1-\alpha x - \beta x^2g(x)-r x^2g(x)^2/g(x)}\\
&=\frac{1}{1-\alpha x - \beta x^2 g(x)-rx^2 g(x)}\\
&=\frac{1}{1-\alpha x - (\beta+r) x^2 g(x)}.
\end{align*}
\end{proof}
\begin{example}
The generating function $g_{3,2}(x)=\mathcal{J}(3,3,3,\ldots;2,2,2,\ldots)$ is that of the (shifted) little Schroeder numbers $s_{n+1}$ \seqnum{A001003}. Then the generating function
$$\MINVERT_2(g_{3,2})(x)=\mathcal{J}(3,3,3,\ldots;4,2,2,\ldots)$$ is the generating function of the central Delannoy numbers $D_n$ \seqnum{A001850}. Note that we have
$$s_{n+1}=\frac{1}{\pi} \int_{3-2\sqrt{2}}^{3+2\sqrt{2}} x^n \frac{\sqrt{-x^2+6x-1}}{4}\,dx,$$
and
$$D_n=\frac{1}{\pi} \int_{3-2\sqrt{2}}^{3+2\sqrt{2}} \frac{x^n}{\sqrt{-x^2+6x-1}}\,dx.$$
\end{example}
\section{Spidernet graphs}
A spidernet graph is a generalization of a tree graph. Spidernet graphs are \emph{stratified}, starting from an \emph{origin point} $o$. A tree graph can be seen as a graph with no edges lying in a same stratum. We follow \cite{HO, Konno, Salimi} for our notation. Thus a graph is a pair $(V, E)$ where $V$ is a set of vertices and  $E \subset \{{\alpha,\beta}|\alpha, \beta \in V, \alpha \ne \beta\}$. Two vertices $\alpha$ and $\beta$ are adjacent if $\{\alpha, \beta\} \in E$ and we write $\alpha \sim \beta$. A finite sequence $\alpha_0, \alpha_1, \ldots, \alpha_n$ is said to be a walk of length $n$ if $\alpha_k \sim \alpha_{k+1}$ for $k=0,1,\ldots,n-1$. We define $\partial(\alpha,\beta)$ as the length of the shortest walk connecting $\alpha$ to $\beta$.  The degree or valency of a vertex $\alpha \in V$ is defined by
$$\kappa(\alpha)=|\{\beta \in V| \beta \sim \alpha\}|.$$
We can stratify $V$ by taking an element of $V$ as origin $o$ and taking
$$V= \bigcup_{k=0}^{\infty} V_k, \quad  \quad V_k=\{\alpha \in V | \partial(o,\alpha)=k\}.$$ For $\alpha \in V$, we set
$$\omega_{\epsilon}(\alpha)=|\{\beta \in V_{k+\epsilon}, \alpha \sim \beta\}|, \quad\quad \epsilon \in \{-1,0,1\}.$$ We also write $\{-,0,+\}$ for $\{-1,0,1\}$.  We consider integers $a, b, c$ such that $a \ge 1, b \ge 2$ and $1 \le c \le b-1$. A \emph{spidernet} $S(a,b,c)$ is a stratified graph $(V, E, o)$ which satisfies the following conditions:
$$\omega_+(o)=a, \omega_0(o)=0, \omega_{-}(o)=0,$$
$$\omega_+(\alpha)=c, \omega_-(\alpha)=1, \omega_0(\alpha)=b-1-c\quad\text{for } \alpha \ne o.$$
We then have $|V_0|=1$ and $|V_k|=ac^{k-1}$ for $k=1,2,3,\ldots$. The fundamental result in this context is that there exists a spectral measure characterized by a generating function
$$g_{a,b,c}(x)=\cfrac{1}{1-\cfrac{ax^2}{1-(b-1-c)x-\cfrac{cx^2}{1-(b-1-c)x-\cfrac{cx^2}{1-(b-1-c)x-\cdots}}}}.$$ Then
$$g_{a,b,c}(x)=\frac{2c}{a \sqrt{x^2(b^2-2b(c+1)+(c-1)^2)-2x(b-1-c)+1}+ax(b-1-c)-a+2c}.$$
\begin{center}
\begin{figure} \label{Fig}
\begin{center}
\includegraphics[height=60mm,width=140mm]{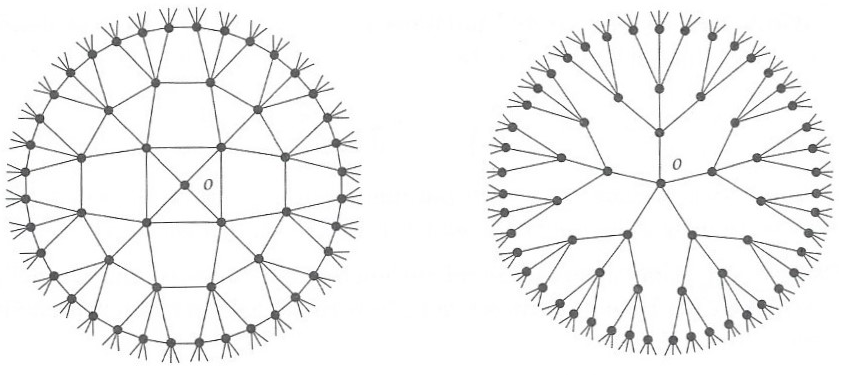}
\caption{Spidernet graphs of $S(4,6,3)$ and $S(5,4,3)$, reproduced from Hora \& Obata \cite{HO} }
\end{center}
\end{figure}
\end{center}
\begin{example} We consider the spidernet $S(4,6,3)$ (see Figure \ref{Fig}). Here, $b-1-c=2$ and hence we have
$$g_{4,6,3}(x)=\cfrac{1}{1-\cfrac{4x^2}{1-2x-\cfrac{3x^2}{1-2x-\cfrac{3x^2}{1-2x-\cdots}}}}.$$
This gives
$$g_{4,6,3}(x)=\frac{1+4x-2 \sqrt{1-4x+8x^2}}{16x^2+8x-1}.$$
We then find that
$$\mu_{4,6,3}(n)=\frac{2}{\pi}\int_{2-2\sqrt{3}}^{2+2\sqrt{3}} x^n \frac{\sqrt{4(x+2)-x^2}}{16+8x-x^2}\,dx+(\sqrt{2}-1)(4-4\sqrt{2})^n,$$
where $g_{4,6,3}(x)=\sum_{n=0}^{\infty} \mu_{4,6,3}(n)x^n$.
Thus we have an absolutely continuous component, and an atomic component coming from the denominator $16+8x-x^2$.
\end{example}
Following \cite{HO}, for three real number $p>0, q\ge 0, a \in \mathbb{R}$, a probability measure $\mu_{p,q,a}$ uniquely specified by 
$$\int_{-infty}^{+\infty} \frac{\mu_{p,q,a}(dx)}{z-x}=\frac{1}{z}\genfrac{}{}{0pt}{0}{}{-}\frac{p}{z-a}\genfrac{}{}{0pt}{0}{}{-}\frac{q}{z-a}\genfrac{}{}{0pt}{0}{}{-}\frac{q}{z-a}\genfrac{}{}{0pt}{0}{}{-\cdots},$$
is called the \emph{free Meixner law} with parameters $(p,q,a)$. The moments for this law will thus have a generating function given by $m_{p,q,a}(x)=\mathcal{J}(0,a,a,a,\ldots;p,q,q,q,\ldots)$. It follows that 
$$m_{p,q,a}(x)=-\frac{p \sqrt{(1-ax)^2-4qx^2}-apx+p-2a}{2(p^2x^2+apx-p+q)}.$$ This expands to give a sequence that begins 
$$1, 0, p, ap, p(a^2 + p + q), ap(a^2 + 2p + 3q), p(a^4 + 3a^2(p + 2q) + p^2 + 2pq + 2q^2),\ldots.$$ The absolutely continuous component of its density is given by
$$\frac{p \sqrt{4q-(a-x)^2}}{2(x^2(p-q)-apx-p^2)}.$$ Thus 
$$m_{p,q,a}(x)=\cfrac{1}{1-\cfrac{px^2}{1-ax-\cfrac{qx^2}{1-ax-\cfrac{qx^2}{1-ax-\cdots}}}}.$$
\begin{example} We consider the generating function  $g_{2,3}(x)$ where 
$$g_{2,3}(x)=\cfrac{1}{1-2x-\cfrac{3x^2}{1-2x-\cfrac{3x^2}{1-2x-\cdots}}}.$$ 
Applying $\MINVERT_1$ to $g_{2,3}(x)$, we obtain the generating function
$$\cfrac{1}{1-2x-\cfrac{4x^2}{1-2x-\cfrac{3x^2}{1-2x-\cdots}}}.$$ Next, we apply $\INVERT_{-2}$ to obtain the generating function 
$$m_{4,3,2}(x)=\cfrac{1}{1-2x-\cfrac{4x^2}{1-2x-\cfrac{3x^2}{1-2x-\cdots}}}.$$ 
Thus 
$$m_{4,3,2}(x)=\INVERT_{-2} \cdot \MINVERT_1(g_{2,3})(x).$$
\end{example}
In a similar way, we have more generally the following result.
\begin{proposition}
$$m_{p,q,a}(x)=\INVERT_{-a} \cdot \MINVERT_{p-q}(g_{a,q})(x).$$
\end{proposition} 
\begin{example} The Riordan numbers \seqnum{A005043} have their generating function given by 
$$r(x)=\frac{m(x)}{1+x m(x)}=\INVERT_{-1}(m)(x).$$ 
We can express their generating function as $\mathcal{J}(0,1,1,1,\ldots;1,1,1,\ldots)$. Thus they are the moments for the density associated to the spidernet $S(1,3,1)$. Their moment representation is given by 
$$R_n=\frac{1}{2\pi} \int_{-1}^3 x^n \frac{\sqrt{-x^2+2x+3}}{1+x}\,dx.$$ 
\end{example}

\begin{example} We have seen that the $\MINVERT_1$ transform of the Motzkin numbers yields the central trinomial numbers. If we now apply an $\INVERT_{-1}$ transform we obtain the sequence \seqnum{A109190}, with generating function $$\INVERT_{-1}\cdot \MINVERT_1 (m)(x)=\frac{1}{1-2x^2 m(x)}=\mathcal{J}(0,1,1,1,\ldots;2,1,1,\ldots).$$ This sequence is thus the moment sequence for the spidernet $S(2,3,1)$. It gives the number of $(1,0)$-steps at level zero in all grand Motzkin paths of length $n$. The moment representation of this sequence is given by 
$$\frac{1}{\pi}\int_{-1}^3 x^n \frac{\sqrt{-x^2+2x+3}}{4+2x-x^2}\,dx+\frac{1}{\sqrt{5}}(1-\sqrt{5})^n.$$ 
\end{example}

\section{A note on Hankel transforms}
For a sequence $(a_n)_{n \ge 0}$ whose generating function can be expressed as a Jacobi continued fraction 
$\mathcal{J}(\alpha_1, \alpha_2, \ldots; \beta_1, \beta_2,\ldots)$ the Hankel transform $h_n$ of $a_n$, defined as the sequence of Hankel determinants $h_n=|a_{i+j}|_{0 \le i,j \le n}$ \cite{Kratt, Kratt1, Layman}, is given by
$$h_n = \prod_{k=1}^n \beta_k^{n-k}.$$
In particular, the moment sequence which is the expansion of $g_{\alpha, \beta}(x)$ will have its Hankel transform given by $\beta^{\binom{n+1}{2}}$ while the moment sequence which is the expansion of $m_{p,q,a}(x)$ will have its Hankel transform given by $p^n q^{\binom{n}{2}}$.
\begin{example} The moment sequence for the $S(4,6,3)$ spidernet has generating function
$$\mathcal{J}(0,2,2,2,\ldots; 4,3,3,3,\ldots.$$  This sequence begins
$$1, 0, 4, 8, 44, 168, 776, 3472, 16204,\ldots.$$ Its Hankel transform begins
$$1, 4, 48, 1728, 186624, 60466176, 58773123072, 171382426877952,\ldots$$ whose general term is given by $4^n3^{\binom{n}{2}}$.
\end{example}
\begin{example} We have seen that the $\MINVERT_1$ transform of the Catalan numbers is the sequence
$$1, 1, 3, 8, 24, 75, 243, 808, 2742, 9458, 33062,\ldots$$ with generating function $\frac{c(x)}{1-x^2c(x)^2}$. This is the convolution of the Catalan numbers $C_n$ and the Fine numbers $\mathcal{F}_n$ \seqnum{A000957} \cite{Fine}.
The Fine numbers have generating function given by the continued fraction
$$\cfrac{1}{1-\cfrac{x^2}{1-2x-\cfrac{x^2}{1-2x-\cfrac{x^2}{1-2x-\cdots}}}}.$$ Note that this is $m_{1,1,2}(x)$, which corresponds to the $S(1,4,1)$ spidernet. The Hankel transform of the convolution is $F_{2n+1}$ \seqnum{A001519}$(n+1)$.

We note that the generating function of the Fine numbers can be expressed as 
$$\frac{2}{1+2x+\sqrt{1-4x}}=\Rev\left(\frac{1-2x}{(1-x)^2}\right),$$ and we have the moment representation 
$$\mathcal{F}_n=\frac{1}{2\pi}\int_0^4 x^n \frac{\sqrt{x(4-x)}}{1+2x}\,dx+\frac{3}{4}\left(-\frac{1}{2}\right)^n.$$
\end{example}

\section{Series inversion}
The reversion of $g_{\alpha, \beta}(x)$ is the rational generating function $\frac{1}{1+\alpha x + \beta x}$. Recall that this follows since the series reversion (compositional inverse) of $xg_{\alpha, \beta}(x)$ is given by $\frac{x}{1+\alpha x + \beta x}$. Similarly the reversion of 
$$m_{p,q,a}(x)=-\frac{p \sqrt{(1-ax)^2-4qx^2}-apx+p-2a}{2(p^2x^2+apx-p+q)}$$ is given by 
$$\Rev(m_{p,q,a})(x)=\frac{p \sqrt{(1-ax)^2+4(p-q)x^2}-apx-p+2q}{2(p^2x^2-apx+q)}.$$ 
This is given by $\frac{1}{x}$ times the solution to the equation 
$$\frac{up \sqrt{(1-au)^2+4(p-q)u^2}-apu-p+2q}{2(p^2u^2-apu+q)}=x.$$ 
This expands to give a sequence that begins 
$$1, 0, -p, - ap, - p(a^2 - 2p + q), - ap(a^2 - 5p + 3q), - p(a^4 + 3a^2(2q - 3p) + 5p^2 - 6pq + 2q^2),\ldots.$$
We have the following result.
\begin{proposition}
$$\Rev(m_{p,q,a})(x)=m_{-p, q-p,a}(x).$$ 
\end{proposition}
Note that we have relaxed the conditions on $p, q$, and $a$. 
\begin{proof}
Our claim is that the reversion $u(x)$ can be expressed as 
$$\frac{u(x)}{x}=\cfrac{1}{1+\cfrac{px^2}{1-ax-\cfrac{(q-p)x^2}{1-ax - \cfrac{(q-p)x^2}{1-ax-\cdots}}}}.$$ 
We solve for $v(x)$ where
$$v(x)=\frac{1}{1-ax-(q-p)x^2 v(x)}=g_{a,(q-p)}(x),$$ and then 
$$u(x)=\frac{1}{1+px^2 g_{a,(q-p)}(x)}.$$
Simplifying shows that 
$$u(x)=\frac{p \sqrt{(1-ax)^2+4(p-q)x^2}-apx-p+2q}{2(p^2x^2-apx+q)}$$ as required.
\end{proof}

\section{Conclusions} Spidernet graphs give rise to interesting Jacobi continued fractions, associated to the free Meixner probability density. Using standard sequence transforms, and the novel ``mean INVERT'' transform, we link moment sequences related to spidernet graphs to other classical moment sequences, such as the Catalan numbers, the Delannoy numbers, and the Fine numbers.

\bigskip
\hrule
\bigskip
\noindent 2020 {\it Mathematics Subject Classification}: Primary
44A60; Secondary 05C10 05C30, 05C63, 05A15, 11B83.
\noindent \emph{Keywords:} Moment sequence, generating function, continued fraction, spidernet graph, integer sequence.

\bigskip
\hrule
\bigskip
\noindent (Concerned with sequences
\seqnum{A000108},
\seqnum{A000957},
\seqnum{A000958},
\seqnum{A001003},
\seqnum{A001006},
\seqnum{A001519},
\seqnum{A001850}
\seqnum{A002426}, 
\seqnum{A005043},
\seqnum{A109190}, and
\seqnum{A111961}).

\end{document}